\documentclass{amsart}

\usepackage{fullpage}
\usepackage{eulervm}
\usepackage{palatino}
\usepackage{verbatim}
\usepackage{graphicx}

\title{A Busy Beaver Problem for Infinite-Time Turing Machines}
\author{Jamie Long and Lee J. Stanley}
\date{October 14th, 2013}

\newcommand{\di}{\displaystyle}
\newcommand{\ita}{\textit}
\newcommand{\n}{\mathbb{N}}

\newtheorem{cor}{Corollary}

\newtheorem{thm}{Theorem}

\theoremstyle{remark}

\newtheorem*{ndef}{Definition}
\newtheorem*{rem}{Remark}

\begin{document}
\maketitle

\begin{abstract}
This note introduces a generalization to the setting of infinite-time computation of the busy beaver problem from classical computability theory, and proves some results concerning the growth rate of an associated function. In our view, these results indicate that the generalization is both natural and promising.
\end{abstract}

\section*{Introduction}
In this note, we formulate an infinite-time analogue of the busy beaver problem from classical computability theory. Before doing so, we briefly recall some preliminary notions from both sorts of computability.

Recent years have witnessed an interest in extending the classical notions of computability (wherein all computations are performed in finite time) to ``infinitary'' computing machines which can perform ``supertasks'' (computations requiring potentially infinite time). Of the supertask machines that have been proposed, perhaps the most prominent are the infinite-time Turing machines (ITTMs), as first described by Hamkins and Lewis in \cite{Hamkins0}. In short, infinite-time Turing machines have ``hardware'' similar to that of classical Turing machines (i.e. they have a small number of one-way infinite tapes, a finite number of states, and a one-cell-wide head for reading and writing elements from a fixed finite alphabet [for our purposes, $\{0,1\}$]). Most importantly, they possess an additional ``Limit'' state which allows them to perform transfinite computations of ordinal length. More concretely, infinite-time Turing machines perform just as their finitary counterparts do during successor steps, while during limit steps, three things happen: (1) each cell is updated by the ``limsup'' convention, (2) the head moves to the left-hand side of the tape, and (3) the machine enters the Limit state.

(It is worth noting here that some sources, such as \cite{welch}, specify that at limit steps, each cell is updated by the ``liminf'' convention. This alternative ITTM model clearly has the same computational power as the original model.)

Rad\'o's busy beaver problem furnished one of the first known instances of explicitly defined functions which are not computable in the classical ``finitary'' sense. The problem is stated as follows: for every $n\in\n$, let $BB$-$n$ denote the collection of $n$-non-halting-state Turing machines which halt after receiving a blank (all-$0$) tape as input, and for every machine $M$ from $BB$-$n$, let $score(M)$ be the number of ones which remain on the tape after such a halting, and $time(M)$ the number of steps performed by $M$ along the way. We are then naturally interested in determining the values of the functions $\Sigma:\n\to\n$ and $S:\n\to\n$  which are defined thusly: \[\Sigma(n)=\max_{M\in BB\text{-}n} score(M)\] \[S(n)=\max_{M\in BB\text{-}n} time(M).\] It is readily seen that $S$ is not computable, as its computability would imply the decidability of the (undecidable) classical halting problem. $\Sigma$ is also noncomputable, but Rad\'o's argument is a bit trickier  (see \cite{rado}); after some appropriate technical modifications, we prove an analogous result in the infinite-time setting.

Finally, let us note that some authors define $BB$-$n$ to be the collection of $n$-non-halting-state Turing machines which, upon receiving a blank (all-$0$) tape as input, ultimately halt with an initial segment of ones at the start of the tape, and zeros in all subsequent cells. Rad\'o's argument that $\Sigma$ is noncomputable carries through in this setting. In the interest of avoiding certain technical issues, it is this definition of $BB$-$n$ which we will ultimately generalize.

\section*{Definitions}
Here, we extend one of the finitary busy beaver functions to the setting of ITTMs, setting out some helpful terminology and notation along the way. For the purposes of this discussion, we will restrict our attention to those ITTMs which possess three tapes, and shall follow \cite{Hamkins0}'s cue in referring to these tapes as ``input,'' ``output,'' and ``scratch.'' 

In line with previous work (e.g. \cite{welch}), we shall assume throughout that all ITTMs follow the unary convention for natural number inputs and outputs; more specifically, we code $n\in\n$ by writing ones to the first $n$ cells of the tape, and zeros on the remaining cells. We denote this coding of $n$ by $\widehat{n}$.

It is of course possible for an ITTM to write an infinite number of ones to its output tape before halting; towards generalizing the busy beaver function, $\Sigma$, it thus is natural for us to restrict our attention to those ITTMs which, upon receiving a blank tape as input, eventually halt with only finitely many ones on the output tape. 

To this end, we make an effective identification of $IT\text{-}\mathcal{C}$, the set of all infinite-time computable functions, with certain (partial) infinite-time computable functions from $\n$ to $\n$: given $f\in IT\text{-}\mathcal{C}$, let $f^{*}$ be the following (partial) infinite-time computable function from $\n$ to $\n$:  \[f^{*}(n)=\begin{cases} k &\text{ if $f(\widehat{n})=\widehat{k}$} \\ \text{undefined} & \text{ if there is no $k$ such that $f(\widehat{n})=\widehat{k}$.}\end{cases}\]

(From the point of view of descriptive set theory, this identification allows us to pass from infinite-time computable functions over a type $1$ product space (namely Cantor space) to infinite-time computable functions over a type $0$ product space (namely $\n$). In doing so, we will be able to appeal to useful results which hold in the type $0$ setting, most notably in our proof of Corollary 2.)

We can now extend $BB$-$n$ and $\Sigma$ to the infinite-time setting.

\begin{ndef}
Let \begin{align*}BB_{\infty}\text{-}n =\{f \in IT\text{-}\mathcal{C}\ |\ & f^{*}(0) \text{ is defined and } \\ & f \text{ is computable by an ITTM with at most } n \text{ non-halting, non-limit states}\}.\end{align*} Then we can define $\Sigma_{\infty}:\n\to\n$ in a natural way: \[\Sigma_{\infty}(n)=\max_{f\in BB_{\infty}\text{-}n} f^{*}(0).\]
\end{ndef}

In analyzing the growth rate of $\Sigma_{\infty}$, we will employ the following commonplace notation.

\begin{ndef}
Let $f,g:\n\to\n$. Then we  shall write ``$f(n)>^{*}g(n)$'' if $f(n)>g(n)$ for all sufficiently large $n\in\n$.
\end{ndef}

\begin{rem}
Colloquially, we say that $f$ ``grows faster than'' $g$. Also, it is worth mentioning that Rad\'o used the notation ``$f(n)>-g(n)$'' instead.
\end{rem}

\section*{Results}
We proceed to demonstrate that our definition of ``infinitary $\Sigma$'' holds promise by generalizing one of Rad\'o's most important (classical) busy beaver results.

\begin{thm}
For every $f \in IT\text{-}\mathcal{C}$ such that $f^*$ is total, $\Sigma_{\infty}(n) >^* f^*(n)$.
\end{thm}
\begin{rem}
Rad\'o demonstrated that for every finitarily computable function $f$, $\Sigma(n)>^{*}f(n)$. Note also that because $\Sigma$ is readily seen to be ITTM-computable, we can establish some astoundingly large lower bounds on $\Sigma_{\infty}$: for instance, $\Sigma_{\infty}(n)>^{*}\Sigma(\Sigma(n))$ and $\Sigma_{\infty}(n)>^{*}(\Sigma(n+42)^{\Sigma(\Sigma(n!))})$, to name but a couple.

(These bounds are all the more impressive in light of the fact that as of 2013, the exact value of $\Sigma(n)$ for $n\geq 5$ remains unknown.)
\end{rem}

Two important corollaries result from this theorem, the first of which is immediately clear, and the second of which requires some familiarity with the elementary definitions and properties of $\Sigma$-pointclasses and $\Sigma$-recursive functions (c.f. \cite{mosch}).

\begin{cor}
$\Sigma_{\infty}$ is not ITTM-computable.
\end{cor}

\begin{cor}
Let $sD$ denote the collection of subsets of natural numbers which are $ITTM$-semi-decidable. Then for every $sD$-recursive $f:\n\to\n$, $\Sigma_{\infty}(n)>^{*}f(n)$. In particular, since $\Pi_{1}^{1}\subseteq sD$ (as established in \cite{Hamkins0}), a similar statement holds with ``$\Pi_{1}^{1}$-recursive'' in place of ``$sD$-recursive.''
\end{cor}
\begin{proof}[Proof of Corollary 2]
Let $f: \n\to\n$ be $sD$-recursive. It suffices to show that $f$ is in fact ITTM-computable. First, observe that since $sD$ is clearly a $\Sigma$-pointclass, the graph of $f$ is ITTM-semi-decidable (c.f. Theorem 3D.2(i) from \cite{mosch}). Then by breaking up the tapes of an ITTM into countably many slices, we can run countably many simultaneous instances of an ITTM-semi-decision procedure for $Graph(f)$ to find the value of $n\in\n$ which is witness to the statement ``$(x,n)\in Graph(f)$.'' (Such an $n$ will in fact exist because $f$, being $sD$-recursive, is total.)
\end{proof}
\begin{rem}
The result which we have cited from \cite{mosch} states that if $\Gamma$ is a $\Sigma$-pointclass, then a function $g:\mathcal{X}\to \n$ is $\Gamma$-recursive if and only if $Graph(g)\in \Gamma$. It is well-known that there are instances where this ``if and only if'' fails if $g$ instead has a codomain which is not a type $0$ product space. In the statement of Corollary 2, it was thus crucial that we specified that $f$ had codomain $\n$. 
\end{rem}

Our proof for Theorem 1 is in the spirit of Rad\'o's work, with extra care taken to account for the special features of infinite-time computation.

\begin{proof}[Proof of Theorem 1]
Let $f\in IT\text{-}\mathcal{C}$ with $f^{*}$ total, and define a function $F: \n\to\n$ by $F(x):=\di\sum_{i=0}^{x} [f^{*}(i) + i]^2$, which is readily seen to be ITTM-computable; since $F$ has codomain $\n$, we can assert that, in particular, $F$ is computable via a $1$-tape ITTM (c.f. \cite{Hamkins1tape}). Fix such an $M$ and let $C$ be the number of its non-halting, non-limit states. By construction, we have that
\begin{enumerate}
\item[(a)] $F(x)\geq f^{*}(x)$ for every $x\in\n$.
\item[(b)] $F(x)\geq x^2$ for every $x\in\n$.
\item[(c)] $F$ is strictly increasing on $\n$.
\end{enumerate}

For the purposes of this proof, let us adopt the convention of taking ``roughly'' to mean ``up to an additive constant depending only on $C$.''

Let $x\in\n$ be fixed, but arbitrary. We claim that there is a 3-tape ITTM with roughly $x$ non-halting, non-limit states which, upon being given a blank input tape, writes a sequence of $F(F(x))$ ones on the left-hand portion of its (otherwise blank) output tape and halts. (In other words, the machine outputs ``$F(F(x))$.'')

The proposed ITTM proceeds as follows:

\begin{enumerate}
\item Using roughly $x$ states, write $x$ ones to the left-hand portion of the input tape.
\item Using roughly $C$ states, run an ``M-like'' machine on the input tape to obtain $F(x)$ ones on the left-hand portion of the (otherwise blank) input tape.
\item Using a constant number of states, copy the contents of the input tape to the scratch tape.
\item Using roughly $C$ states, run an ``M-like'' machine on the scratch tape to write $F(F(x))$ ones on the left-hand side of the (otherwise blank) scratch tape.
\item Using a constant number of states, copy the contents of the scratch tape to the output tape to obtain  $F(F(x))$ ones on the left-hand portion of the (otherwise blank) output tape.
\item Halt.
\end{enumerate}

Since this ITTM possesses $s(x):=x+h(C)$ states and outputs $F(F(x))$, we know from the definition of $\Sigma_{\infty}$ that $\Sigma_{\infty} (s(x))\geq F(F(x))$. Also, since $x^2>^{*}x+h(C)=s(x)$ and $F(x)\geq x^2$, it follows that $F(x)>^{*}s(x)$. Thus, by the monotonicity of $F$ on $\n$, $F(F(x))>^{*}F(s(x))$, and so $\Sigma_{\infty}(s(x))>^{*}F(s(x))>^{*} f^{*}(s(x))$ (by (a) above), whence $\Sigma_{\infty}(n)>^{*}f^{*}(n)$.
\end{proof}
\begin{rem}
The result from \cite{Hamkins1tape} that we cited in the preceding proof states that every ITTM-computable function $f$ with codomain $\n$ can be computed via an ITTM with just one tape. It is worth noting that it is not true that \ita{every} ITTM-computable function $f$ is ITTM-computable via a 1-tape ITTM; in order to be able to compute a general ITTM-computable function $f$ , two tapes might be required, or alternatively, one tape augmented with some sort of auxiliary flag or scratch cell. 
\end{rem}

\section*{Conclusions and Future Work}
We have seen that our $\Sigma_{\infty}$ function generalizes Rad\'o's $\Sigma$ function in a natural fashion, and on the basis of the results which we have presented here, we conclude that this infinitary analogue of the busy beaver problem is worthy of further study. In particular, it would be interesting to investigate the question of for which $\Sigma$-pointclasses $\Gamma$, with $sD\subset \Gamma \subset \Delta_{2}^{1}$, we can prove the analogue of Corollary 2 with ``$\Gamma$-recursive'' replacing ``$sD$-recursive.''

\bibliographystyle{amsalpha}
\nocite{*}
\bibliography{BusyBeaverWriteup}
\end{document}